\documentclass{article}
\usepackage[OT2,OT1]{fontenc}
\usepackage{amsfonts,amsmath, amssymb,latexsym,mathrsfs,comment,url}
\usepackage[all,cmtip]{xy}
\usepackage{float}
\def\Z{{\mathbb Z}}
 \DeclareFontFamily{U}{wncy}{}

\def\ln{{\rm ln}}

\def\ur{{\rm ur}}

\def\Cl{{\rm Cl}}

\def\Disc{{\rm Disc}}

\def\Aut{{\rm Aut}}

\def\F{{\mathbb F}}

\def\Hom{{\rm Hom}}
\def\Q{{\mathbb Q}}

\def\Z{{\mathbb Z}}

\def\F{{\mathbb F}}
\def\Q{{\mathbb Q}}

\newcommand*{\K}{\mathbb{K}}

\DeclareMathOperator{\Sel}{Sel}

\newcommand*{\ra}{\rightarrow}

\newcommand*{\ol}{\overline}

\def\Disc{{\rm Disc}}
\newtheorem{tthm}{Theorem}
\newtheorem{rrem}[tthm]{Remark}
\newtheorem{theorem}{Theorem}[section]

\newtheorem{lemma}[theorem]{Lemma}
\newtheorem{conj}[theorem]{Conjecture}
\newtheorem{remark}[theorem]{Remark}
\newtheorem{proposition}[theorem]{Proposition}
\newenvironment{proof}{\noindent {\bf Proof:}}{$\Box$ \vspace{2 ex}}

\DeclareFontFamily{U}{wncy}{}
 \DeclareFontShape{U}{wncy}{m}{n}{<->wncyr10}{}
\DeclareSymbolFont{mcy}{U}{wncy}{m}{n}
 \DeclareMathSymbol{\Sh}{\mathord}{mcy}{"58}

\usepackage{amscd,amssymb,amsmath}
\usepackage{color}

\author{Arul Shankar and Jacob Tsimerman}
\title{Non-trivial bounds on 2, 3, 4, and 5-torsion in class groups of number fields, conditional on standard $L$-function conjectures}
\begin{document}
\maketitle
\begin{abstract}
We prove new conditional bounds on the $m$-torsion subgroups of class groups of number fields of any fixed degree, for $m=2$, $3$, $4$, and $5$. Our methods first recast the problem in the language of class groups of Galois modules, which allows us to relate these torsion subgroups to Selmer groups of elliptic curves. We then obtain a global estimate using the refined BSD conjecture, in a similar way to how one normally uses the Brauer-Siegel bound. 

Our methods are potentially very general, but rely on the existence of motives with very special $\Z/m\Z$-cohomology. In particular, the restriction to $m\leq 5$ stems from needing an elliptic curve over $\Q$ with $m$-torsion subgroup isomorphic to $\Z/m\Z\oplus\mu_m$.
\end{abstract}

\section{Introduction}

\subsection*{History and main results}

Let $n,m\geq 2$ be fixed integers. For
any number field $K$ of degree $n$, it is conjectured that the size
$h_m(K)$ of the $m$-torsion subgroup of the class group of $K$
satisfies the bound $h_m(K)=O_{m,n}(|\Disc(K)|^\epsilon)$ for
every $\epsilon>0$. This conjecture is currently known only for $n=m=2$
(using Gauss's genus theory).  For other degrees and primes,
the best that is known in general is the bound
\begin{equation}\label{eqBS}
h_m(K)=O_{n,\epsilon}(|\Disc(K)|^{1/2+\epsilon}),
\end{equation}
which follows from the Brauer--Siegel theorem bounding the size $h(K)$
of the entire class group, and applying the trivial inequality
$h_m(K)\leq h(K)$. It has become a problem of much interest to obtain any
improvement on the trivial bound \eqref{eqBS} for any pair $(n,m)\neq
(2,2)$.

The first such result was obtained independently by Pierce
\cite{Pierce} and Helfgott--Venkatesh \cite{HV}, who consider the
case $n=2$ and $m=3$, and prove the bounds
$O_\epsilon(|\Disc(K)|^{27/56+\epsilon})$ and $O_\epsilon(|\Disc(K)|^{0.44178...})$,
respectively. Subsequently, these bounds were improved to
$O_\epsilon(|\Disc(K)|^{1/3+\epsilon})$ by work of Ellenberg and
Venkatesh~\cite{EV}. In this work, they also achieve the bound
$O_\epsilon(|\Disc(K)|^{1/3+\epsilon})$ for the cases $(n,m)=(3,3)$
and $(4,3)$.
Recent work of Bhargava, Taniguchi, Thorne, Zhao and the two authors
\cite{BSTTTZ} obtain power saving bounds $O(|\Disc(K)|^{1/2-1/(2n)+\epsilon})$ for the case $m=2$ for all
degrees $n>2$.  Further unconditional results are also known when the
number fields $K$ are constrained by having a fixed Galois group: in
the work of Wang \cite{Wang1}, Kluners and Wang \cite{KW}, and Wang
\cite{Wang2}, power saving bounds are respectively obtained for
$p$-torsion in the class groups of $K$, where $K$ has Galois group
$(\Z/q\Z)^r$ (here $q$ is prime and $r>1$), where the Galois group of
$K$ is a $p$-group, and where the Galois group of $K$ is a nilpotent
group $G$ such that every Sylow subgroup of $G$ is non-cyclic and
non-quaternion. Additionally, conditional on the generalized Riemann hypothesis, Ellenberg and Venkatesh prove a bound of 
$O(|\Disc(K)|^{\frac12- \frac1{2m(n-1)}+\epsilon})$ for all pairs $(n,m)$.

The purpose of this paper is to introduce a new method for obtaining
such bounds, albeit a conditional one at the moment.  Our main result is the following.
\begin{tthm}\label{thMain5}
Let $n>1$ be a fixed integer, and let $m=2$, $3$, $4$, or $5$.
Assume the Hasse--Weil conjecture, the refined Birch and Swinnerton-Dyer conjecture, and the generalized Riemann hypothesis for $L$-functions of elliptic curves over degree-$n$ number fields. 
Then $h_m(K)=O_{n,\epsilon}(D^{\frac14+\epsilon})$ for any degree-$n$ number field $K$.
\end{tthm}

In the case of quadratic fields, we can get a non-trivial result without assuming either the Hasse--Weil conjceture or the generalized Riemann hypothesis:

\begin{tthm}\label{thm: mainquadratic}
    Let $m=4$ or $5$.
Assume the refined Birch and Swinnerton-Dyer conjecture for elliptic curves over $\Q$. Then $h_m(\Q(\sqrt{D}))=O_{\epsilon}(D^{\frac12-\frac{\delta}2+\epsilon})$ where $\frac12-\delta$ is the best subconvex bound we have for $L$-functions of elliptic curves over $\Q$.\footnote{The best such bound the authors are aware of is $\delta=\frac{25}{256}$ (see \cite{Maga}).}
\end{tthm}

\begin{rrem}{\rm 
We make the following three observations regarding Theorem \ref{thMain5} and its proof.
\begin{itemize}
\item[{\rm (a)}] To our knowledge, this is the first (even conditional) result which gives pointwise bounds, where the quality of the exponent does not degrade with the degree $n$ of the number field. 

\item[{\rm (b)}] To obtain the conclusion of Theorem \ref{thMain5}, we need to assume something weaker than the three conjectures for all elliptic curves $E/K$. In fact, we only use base changes of a single elliptic curve.

\item[{\rm (c)}] If, for a fixed elliptic curve $E/\Q$, we had a way of bounding the rank of $E/K$ by $o(\log|\Disc(K)|)$ unconditionally, we would not need to assume the generalized Riemann Hypothesis for the $L$-function of $E/K$. Assuming the Lindel\"{o}f conjecture for this $L$-function would be sufficient. Moreover, any subconvex estimate on the central value of this $L$-function would also imply a power saving bound on the torsion in the class group.
\end{itemize}

}\end{rrem}

\subsection*{Method of proof}

Our perspective is as follows.
\begin{enumerate}
    \item  First, for an integer $m$, we express the group
    $\Cl(K)[m]$ as an unramified cohomology group of a finite
    Galois module $M$.  In fact, $M$ is simply $\Z/m\Z[\Hom(K,\bar{\Q})]$. We denote this unramified cohomology group by $\Cl(M)$.
    \item Our main idea (philosophically, and somewhat imprecisely) is to find other `motives' $X$ whose mod $m$-reduction is (essentially) isomorphic to a power of $M$. Then we may hope to reinterpret $\Cl(M)$ as the $m$-torsion of a motivic class group $Cl(X)$. There should be a global class number formula for $Cl(X)$ with a trivial bound that one can use to bound $Cl(M).$ In a sense, we are obtaining better ``trivial bounds'' for $\Cl(M)$ by embedding $M$ into different motives.
    \item Concretely, we find an elliptic curve $E$ over $\Q$ such that $E[m]$ is isomorphic to $\mu_m\oplus(\Z/m\Z)$ as Galois modules over $\Q$. It is known that such elliptic curves only exist when $m$ is $2$, $3$, $4$, and $5$, hence our results are limited to these four values of $m$. Our motive is essentially the Abelian variety $A=\textrm{Res}_{K/\Q} E_K$ and its class group corresponds to the Tate-Shafarevich group. The torsion group scheme $A[n]$ is isomorphic to $M\oplus M\otimes \mu_m$, and we show that $\Cl(M)$ and $\Cl(M\otimes\mu_m)$ are approximately the same size (in fact they are dual to each other up to small error).
    
    \item The Refined BSD conjecture plays the role of the class number formula which allows us a method for computing the size of $\Sh_A$, assuming we have control on the rank of $A$ and the size of the central value of the $L$-function. This is where the generalized Riemann hypothesis enters. We then use GRH to provide an upper bound for $\Sh_A$. This in turn provides a trivial upper bound for $\Sh_A[m]$,  which we show is essentially the same size as $\Cl(A[m])$ (once the size of the rank has been controlled). Combining these observations gives an upper bound for $\Cl(M)$.
\end{enumerate}

\noindent In fact, we find it easier to work with $E/K$ directly and not form the Weil-restriction down to $\Q$, but this is a minor point done for technical convenience.

\subsection*{Acknowledgements}
We are happy to thank  Ashay Burungale, Jordan Ellenberg, Ananth Shankar, Joe Silverman, and Tonghai Yang for interesting comments and corrections on a previous version of this paper.

\section{Estimating the terms in the refined BSD conjecture}

Let $E$ be an elliptic curve over $\Q$, and let $K$ be a degree-$n$ number field. Let $L(E/K,s)$ denote the Hasse--Weil $L$-function of $E$ over $K$, normalized so that the critical line is $\Re(s)=1$. We denote the completed $L$-function of $E/K$ by $\Lambda(E/K,s)$:
\begin{equation*}
\Lambda(E/K,s):={\rm Norm}_{K/\Q}(N_{E/K})^s|\Disc(K)|^{2s}\Gamma_K(s)L(E/K,s),
\end{equation*}
where $N_{E/K}$ is the conductor ideal of $E$ over $K$ and $\Gamma_K(s):=((2\pi)^{-s}\Gamma(s))^n$.
We begin by recalling three standard conjectures regarding $E/K$ and its $L$-function.
First, the Hasse--Weil conjecture (HWC) for $E/K$ is the following:
\begin{conj}\label{conj:HW}
The function $L(E/K,s)$ has an analytic continuation to the entire complex plane, and satisfies the function equation
\begin{equation*}
\Lambda(E/K,s)=w(E/K)\Lambda(E/K,2-s),
\end{equation*}
where $w(E/K)$ denotes the root number of $E/K$.
\end{conj}
Denote the algebraic rank of $E/K$ by $r_{E/K}$. The refined Birch and Swinnerton-Dyer (BSD) conjecture for $E/K$ is the following:
\begin{conj}\label{conj:BSD}
The algebraic rank $r_{E/K}$ is equal to the order of vanishing of $L(E/K,s)$ at $s=1$. Moreover, the leading term of the Taylor series of $L(E/K,s)$ at $s=1$ is given by
\begin{equation}\label{eq:BSD}
\lim_{s\to 1}\frac{L(E/K,s)}{(s-1)^{r_{E/K}}}
=
|\Disc(K)|^{-1/2}\cdot\#\Sh_{E/K}
\cdot R_{E/K}\cdot\frac{\Omega_{E/K}C_{E/K}}{\#E(K)_{{\rm tor}}^2}.
\end{equation}
Above, $\Disc(K)$ denotes the discriminant of $K$, $\Sh_{E/K}$ is the Tate--Shafarevich group of $E$ over $K$,  $R_{E/K}$ is the regulator of $E$ over $K$, $\Omega_{E,K}:=\Omega_+^{r_1}\Omega_-^{r_2}$ is the product of the two periods of $E$ raised to the number of real and complex embeddings of $K$, respectively, and $C_{E/K}$ is given by
\begin{equation*}
C_{E/K}:=\prod_v c_v(E/K)
\Big|
\frac{\Delta_{E,v,\min}}{\Delta_{E,v}} 
\Big|_v^{\frac1{12}},
\end{equation*}
where the product is over all finite places $v$ of $K$, and where $\Delta_{E,v,\min}$ denotes the discriminant of a minimal model for $E$ over $K_v$. 
\end{conj}
Finally, we have the Grand Riemann Hypothesis (GRH) for the $L$-function of $E/K$.
\begin{conj}\label{conj:GRH}
The nontrivial zeroes of $L(E/K,s)$ lie on the line $\Re(s)=1$.
\end{conj}
In the remainder of the section, we obtain bounds on the rank $r_{E/K}$ of $E/K$ and the size of the Tate--Shafarevich group $\Sh_{E/K}$ conditional on 
the above three conjectures on $E/K$.
Our bounds are in the setting where $E$ and the degree $n$ of the number field $K$ are fixed, while $K$ is varying.

\medskip

We begin with the following lemma.
\begin{lemma}\label{lem:junkbound}
We have
\begin{equation*}
1\ll_{n,E} \frac{\Omega_{E/K}C_{E/K}}{\#E(K)_{{\rm tor}}^2}\ll_{n,E,\epsilon} 1.
\end{equation*}
\end{lemma}
\begin{proof}
The quantity $\Omega_{E/K}$ depends only on $E$ and the number of real and complex embeddings of $K$. By Merel's generalization \cite{Merel} of Mazur's theorem \cite{MazurTor}, it follows that we have 
$\#E(K)_{{\rm tor}}\ll_n 1$. We also have the inequalities
\begin{equation*}
1\leq\prod_{v} \Big |\frac{\Delta_{E,v,\min}}{\Delta_E}\Big |_v\ll_{n,E}1,\quad\quad
1\leq \prod_{v} c_v(E/K)\ll_{n,E} 1,
\end{equation*}
where the first two inequalities are immediate, and the second two follow from an argument identical to 
the proof of \cite[Lemma 6.2.1]{PPVW}.
\end{proof}

Next, we we have the following conditional bounds on the rank of $E/K$ from \cite[Proposition 5.21]{IK}.
\begin{lemma}\label{lem:rankbound}
Assume that $E/K$ satisfies the HWC and the BSD conjecture, and that $L(E/K,s)$ satisfies GRH. Then we have
\begin{equation*}
r_{E/K}\ll_{n,E} \frac{\log |\Disc(K)|}{\log \log |\Disc(K)|}.
\end{equation*}
\end{lemma}

Next, we obtain a lower bound on the regulator depending on the rank of $E/K$:
\begin{lemma}\label{lem:regulatorbound}
With notation as above, we have 
\begin{equation*}
R_{E/K}\gg_{n,E} \Bigl(
\frac{\log |\Disc(K)|}{r_{E/K}}
\Bigr)^{r_{E/K}}.
\end{equation*}
\end{lemma}
\begin{proof}
Let $P=(x,y)\in E(K)$ be a rational point. Let $F=\Q(x,y)$ be the subfield of $K$ generated by $x$ and $y$. We know that the canonical height $\hat{h}(P)$ of $P$ is within $O_E(1)$ of the Weil height of $P$ which is simply the logarithmic height of $x\in\bar{Q}$. Since $x$ generates $F$, it follows that we have $\hat{h}(P)\gg_E\log|\Disc{F}|$. Moreover, by Lemma \ref{lem:rankbound}, we know that the rank of $E$ over $F$ is $\ll_{n,E}\log|\Disc(F)|$. The lemma now follows from Minkowski's second theorem.
\end{proof}

\begin{remark}{\rm
We note that sophisticated lower bounds for the canonical heights of nontorsion points in elliptic curves $E/K$ are proven by Silverman \cite{Si} and Hindry--Silverman \cite{HS}. However, those results are proven in the setting where $K$ is fixed and $E$ is varying, while we need bounds in the setting where $E$ is fixed and $K$ is varying.
}\end{remark}

The conductor of $L(E/K,s)$ is $\Disc(K)^2{\rm Norm}_{K/\Q}(A(E/K))\ll_E \Disc(K)^2$. This yields the ``convex'' bound $|L(E/K,s)|\ll_{E,\epsilon} |\Disc(K)|^{1/2+|s|+\epsilon}$, for complex numbers $s$ sufficiently close to $1$. In our next result, we bound the leading Taylor coefficient of the $L$-function $L(E/K,s)$, assuming a rank bound on $E/K$ and a subconvex bound on $|L(E/K,s)|$.
\begin{lemma}\label{lem:Lsubconvex}
Assume that we have $|L(E/K,s)|\ll_{E,\epsilon} |\Disc(K)|^{1/2+|s|-\eta+\epsilon}$ for some $\eta>0$, and that the rank of $E$ over $K$ is bounded by $r_{E/K}\ll_{E}\log|\Disc(K)|/\log\log|\Disc(K)$. Then we have
\begin{equation*}
\lim_{s\to 1}\frac{L(E/K,s)}{(s-1)^{r_{E/K}}}
\ll_{E,\epsilon} |\Disc(K)|^{1/2-\eta+\epsilon}.
\end{equation*}
\end{lemma}
\begin{proof}
Denote $r_{E/K}$ by $r$ and note that the result is immediate when $r=0$. When $r\geq 1$, we write
\begin{equation*}
\lim_{s\to 1}\frac{L(E/K,s)}{(s-1)^{r}}=L^{(r)}(E/K,1)=r!\cdot {\rm Res}_{s=0}\frac{L(E/K,1+s)}{s^{r+1}}.
\end{equation*}
Therefore, from the residue theorem, we have
\begin{equation*}
L^{(r)}(E/K,1)=
\frac{1}{2\pi i} \int_{|s|=\theta} \frac{L(E/K,1+s)}{s^{r+1}} ds
\end{equation*}
for every $\theta>0$. Taking $\theta=(\log\log\log |\Disc(K)|)^{-1}$ and using the upper bound on $r$ yields
\begin{equation}\label{eqLsubconvex}
L^{(r)}(E/K,1)
\ll_{E,\epsilon} |\Disc(K)|^{\frac 12-\eta+\epsilon},
\end{equation}
as necessary.
\end{proof}

Finally, we combine \eqref{eq:BSD} with Lemmas \ref{lem:junkbound}, \ref{lem:rankbound}, \ref{lem:regulatorbound}, and \ref{lem:Lsubconvex}, to obtain the following conditional upper bound on the Tate--Shafarevich group $\Sh_{E/K}$.

\begin{proposition}\label{thmsha}
Let $E$ be a fixed elliptic curve over $\Q$ and let $n\geq 2$ be a fixed integer. Let $K$ be a number field of degree $n$, and assume Conjectures \ref{conj:HW} $($HWC$)$, \ref{conj:BSD} $($BSD$)$, and \ref{conj:GRH} $($GRH$)$ for $E/K$ and $L(E/K,s)$. Then we have
\begin{equation*}
\#\Sh_{E/K}\ll_{n,E,\epsilon} |D|^{1/2+\epsilon}.
\end{equation*}
\end{proposition}

\section{Comparing torsion in $\Cl(K)$ to torsion in $\Sh_{E/K}$}

\subsection{Cohomology of finite Galois modules over number fields}
Let $A$ be a finite Galois module over a field $F$, i.e., a finite
abelian group $A$, along with an action of the absolute Galois group
$G_F$ of $F$ on $A$. We let $H^i(F,A)$ denote the $i$'th cohomology
group arising from this action.
Now let $A$ be a finite Galois module over a local field $K_v$ (with residue field denoted $\F_v$).  Let $I_v$ denote
the intertia subgroup of $G_{K_v}$. We say that $A$ is {\it
  unramified} if $I_v$ acts trivially on $A$. When $A$ is unramified,
there is a natural action of $G_{K_v}/I_p\cong G_{\F_v}$ on $A$, and
we have an injection
\begin{equation*}
H^1(\F_v,A)\hookrightarrow H^1(K_v,A).
\end{equation*}
Define the subgroup of unramified cohomology classes by
\begin{equation*}
H^1_\ur(K_v,A):=\ker\bigl(H^1(K_v,A)\to H^1(I_v,A)\bigr).
\end{equation*}
When $A$ is unramified, it follows from the inflation-restriction
exact sequence that the subgroup of unramified cohomology classes is
exactly given by the image of $H^1(\F_v,A)$ (see \cite[Lemma
  3.2]{Rubin}).

Next, let $A$ be a finite Galois module over a number field $K$. Let $v$ be a place of $K$. The absolute Galois
group $G_{K_v}$ of $K_v$ naturally injects into $G_K$, yielding natural maps $H^1(K,A)\to H^1(K_v,A)$. When $v$ is a finite place, we say that $A$ is {\it unramified at $v$} if $A$ is unramified as a $K_v$-Galois module. We define the class group $\Cl(A)$ of $A$ to
be the set of elements $\sigma\in H^1(K,A)$ such that for every $v$
at which $A$ is unramified, the image of $\sigma$ in $H^1(K_v,A)$
is an unramified cohomology class. That is
\begin{equation*}
  \Cl(A):=\ker\Bigl(H^1(K,A)\to
  \sideset{}{'}\prod_v H^1(K_v,A)/H_\ur^1(K_v,A)\Bigr),
\end{equation*}
where the product is over all finite places $v$ at which $A$ is unramified.

\subsection{Selmer groups of elliptic curves}

Let $E$ be an elliptic curve over a field $F$ and let $m$ be a
positive integer. We have the exact sequence
\begin{equation*}
0\to E[m]\to E\to E\to 0,
\end{equation*}
of $G_F$-modules, where
$E[m]$ is the $m$-torsion subgroup of $E$. This gives rise to the
following exact sequence of Galois cohomology groups:
\begin{equation}\label{eqgcseq}
0\to E(F)/mE(F)\to H^1(F,E[m])\to H^1(F,E)\to 0.
\end{equation}

We now specialize to the arithmetic case. Let $E$ be an elliptic curve
over a number field $K$. We have the exact sequence \eqref{eqgcseq} for the fields
$F=K$ and $F=K_v$ for all completions $K_v$ of $K$. 
The {\it $m$-Selmer
  group} $\Sel_m(E)$ is then defined to be
\begin{equation}\label{eqselECdef}
\Sel_m(E):=\ker\Bigl(H^1(K,E[m])\to\prod_v H^1(K_v,E)\Bigr).
\end{equation}
Above, the product is over all places $v$ of $K$, and the map
$H^1(K,E[m])\to H^1(K_v,E)$ is given by composing the map
$H^1(K,E[m])\to H^1(K_v,E[m])$ with the map $H^1(K_v,E[m])\to
H^1(K_v,E)$ of \eqref{eqgcseq}.

Our next results compare the $m$-Selmer group of $E$ with the class group of the finite Galois module $E[m]$. We start with
the following lemma.
\begin{lemma}\label{lemCass}
Let $E$ be an elliptic curve over a number field $K$, and let $m>1$ be an
integer. Let $v$ be a place not dividing $m$ at which $E[m]$ is
unramified and at which $E$ has good reduction. Then
\begin{equation*}
\ker\left(H^1(K_v,E[m])\to H^1(K_v,E)\right)=H_\ur^1(K_v,E[m]).
\end{equation*}
That is, the image of a class $\sigma\in H^1(K_v,E[m])$ in
$H^1(K_v,E)$ is soluble at $p$ if and only if it is unramified.
\end{lemma}
\begin{proof}
This is precisely the content of \cite[Lemma 19.3]{Cassels}.
\end{proof}

\begin{proposition}\label{propellip}
Let $E$ be an elliptic curve over $K$ of degree $n$, and let $m>1$. Let $s$ be the number of places of $K$ dividing $m\Delta(E)$. Then we have
\begin{equation*}
\log\Bigl(\frac{\#\Sel_m(E)}{\#\Cl(E[m])}\Bigr)=O_{m,n}(s).
\end{equation*}
\end{proposition}
\begin{proof}
Both the groups $\Sel_m(E)$ and $\Cl(E[m])$ are subgroups of the Galois cohomology group $H^1(K,E[m])$. The Galois module $E[m]$ is unramified away from the prime dividing $m\Delta(E)$. It follows from Lemma~\ref{lemCass} that the Selmer structures defining $\Sel_m(E)$ and $\Cl(E[m])$ differ at most at $s+n$ `bad' local places (the $n$ coming from the infinite place of $K$). Thus $\log\frac{\#\Sel_m(E)}{\#\Cl(E[m])}$ is bounded by $\ln\prod_{v \textrm{ bad}} |H^1(K_v,E[m])|$. We are thus finished if we prove that $|H^1(K_v,E[m])|= O_{m,n}(1)$.

Let $N$ denote the subgroup of $G_{K_v}$ which fixes $E[m]$. Then by the inflation-restriction exact sequence, we have 
\begin{equation*}
\begin{array}{rcl}
|H^1(K_v,E[m])|&\leq&|H^1(G/N,E[m])|\cdot 
|H^1(N,E[m])|
\\[.05in]&=&
|H^1(G/N,E[m])|\cdot |\Hom(N,E[m]|.
\end{array}
\end{equation*}
The first term in the second line of the above equation is $O_{m,n}(1)$ by virtue of having finitely many possibilities for $G/N$ and its action. On the other hand $|\Hom(N,E[m])|=|N[m]|$ is bounded by $m$ times the number of degree $m$ extensions of $L=\ol{K_{v}}^N$ and hence by the number of degree $\leq d$ extensions of $\Q_p$, where $d=mn\cdot \#\Aut E[m]$. This is well-known to be uniformly bounded. 
\end{proof}

\subsection{Class groups of number fields}

Let $K$ be a  number field of degree $n$ over $\Q$, and let $m>1$ be a (fixed) integer. Class field theory provides a bijection between index-$m$ subgroups of $\Cl(K)$ and degree-$m$ abelian extensions of $K$ which are unramified at every place. Let $v$ be a place of $K$. The set of degree-$m$ abelian extensions of $K$ (resp.\ $K_v$) is parametrized by $H^1(K,\Z/m\Z)$ (resp.\ $H^1(K_v,\Z/m\Z)$). Let $L$ be the extension of $K$ corresponding to $\sigma\in H^1(K,\mu_m)$. Then $L_v:=L\otimes_K K_v$ corresponds to the image of $\sigma$ under the natural map $H^1(K,\Z/m\Z)\ra H^1(K_v,\Z/m\Z)$.  
Furthermore, the extension $L_v$ of $K_v$ is unramified if and only if it corresponds to an unramified class in $H^1(K_v,\Z/m\Z)$. Therefore, the size of the $m$-torsion of the class group of $K$ is
equal to the size of $H^1_\ur(K,\Z/m\Z)=\Cl(\Z/m\Z)$.

Meanwhile, we have the isomorphism $H^1(K,\mu_m)\cong K^\times/(K^\times)^m$ and, for a finite place $v$ of $K$ not dividing $m$, a class in $H^1(K,\mu_m)$ is unramified at $v$ if and only if it corresponds to an element in $K^\times/(K^\times)^m$ whose valuation at $v$ is divisible by $m$. Therefore, defining the group $A_{m,K}$ by 
$$A_{m,K}\cong \Big(\ker K^\times/(K^\times_m)\ra \prod_v \Z/m\Z\Big),$$
we see that $|\Cl(\mu_m)|$ is within a factor of $O_{m,n}(1)$ of $|A_{m,K}|$.

There is a natural map $\phi:A_{m,K}\ra \Cl(K)[m]$ given by $\phi(\alpha)=(\alpha)^{\frac1m}$ which is surjective with kernel isomorphic to $U_K/U_K^m$, the $m$-cotorsion of the unit group of $K$. We thus obtain the following result:

\begin{proposition}\label{prop:mumzm}
Let $K$ be a number field of degree $n$, and consider the $G_K$ modules $\mu_m, \Z/m\Z$. The group $\Cl(\Z/m\Z)$ is isomorphic to the dual of $\Cl(K)[m]$, 
and $\Cl(\mu_m)$ has the same size as $\Cl(K)[m]$ up to a factor of ${O_{m,n}(1)}$.
\end{proposition}

\section{Proofs of Theorem \ref{thMain5} and \ref{thm: mainquadratic}}

We begin with the following lemma.
\begin{lemma}
Let $m\in\{2,3,4,5\}$ be fixed. Then there exists an elliptic curve $E/\Q$ such that we have the isomorphism
\begin{equation*}
E[m]\cong \mu_m\oplus\Z/m\Z
\end{equation*}
as group schemes.
\end{lemma}
\begin{proof}
Since the modular curves $\Gamma(m)$ have genus $0$ for each $m\in\{2,3,4,5\}$, there do exist such curves. Below, we provide specific examples. 
\begin{equation*}
\begin{array}{rl}
m=2. & E:y^2=x^3+x\\[.03in]
m=3. & E:y^2+y=x^3\\[.03in]
m=4. & E:y^2+xy=x^3-4x-1\\[.03in]
m=5. & E:y^2+y=x^3-x^2-10x-20
\end{array}
\end{equation*}
These examples have been taken from the database LMFDB \cite{LMFDB}.
\end{proof}

Let $m$ and $E$ be as in the above lemma, and let $K$ be a number field of degree $n$. Denote the elliptic curve $E$ over $K$ by $E_K$. We then have the isomorphism
\begin{equation*}
E_K[m]\cong \mu_m\oplus\Z/m\Z \end{equation*}
as group schemes. We next prove the following result.


\begin{proposition}\label{selclsq}
The quantities $\#\Sel_m(E_K)$ and $\#(\Cl(K)[m])^2$ are within a factor
of $O_\epsilon(|\Disc(K)|^\epsilon)$ of each other.
\end{proposition}
\begin{proof}
By Proposition \ref{propellip}, we see that
$|\log\#\Cl(E_K[m])-\log\#\Sel_m(E_K)|$ is bounded by $O_{n,E}(1)$. By
Proposition \ref{prop:mumzm}, we see that
$\#\Cl(\Z/m\Z)=\#\Cl(K)[m]$ and that $|\log\#\Cl(\mu_m)-\log\#\Cl(K)[m]|$ is bounded by $O_{n,m}(1)$. The result follows from the structure of $E_K[m]$.
\end{proof}

We are ready to prove the main results.

\medskip

\noindent\textbf{Proof of Theorem \ref{thMain5}:}
From the exact sequence
$$0\ra \frac{E_K(K)}{mE_D(K)} \ra \Sel_m(E_K)\ra \Sh(E_K)[m]\ra 0$$
it follows that $$\#\Sel_m(E_K)=\#\frac{E_K(K)}{mE_K(K)}\cdot\#\Sh(E_K)[m].$$ 
Our conditional rank
bound from \ref{lem:rankbound} implies that $\#\frac{E_K(K)}{mE_K(\K)}=O_\epsilon(|\Disc(K)|^\epsilon)$. Theorem \ref{thMain5} now follows immediately by combining Proposition \ref{selclsq} with the conditional bound on $\Sh(E_K)[m]$ that was
proved in Theorem~\ref{thmsha}. $\Box$

\medskip

\noindent \textbf{Proof of Theorem \ref{thm: mainquadratic}:}
let $K=\Q(\sqrt{D})$ be a quadratic field. Then the Weil restriction of $E_K$ to $\Q$ is isomorphic to the abelian surface $E\oplus E_D$, where $E_D$ denotes the twist of $E$ by $D$. Thus the Hasse--Weil conjecture for $E_K$ is known by the modularity of $E$ and $E_D$.

Next, denote the Mordell-Weil rank of $E_D$ by $r_D$. Let $\Sel_2(E_D)$ denote
the $2$-Selmer group of $E_D$. From the exact sequence
\begin{equation*}
0\to\frac{E_D(\Q)}{2E_D(\Q)}\to\Sel_2(E_D)\to \Sh(E_D)[2]\to 0
\end{equation*}
we see that $r_D\leq\log_2\#\Sel(E_D)$.  Since the field
$\Q(E_D(\bar{\Q})[2])$ is the same as $\Q(E(\bar{\Q})[2])$, it follows
from \cite[Proposition 7.1]{BrKr} and \cite[Proposition 9.8]{Mazur1}
that we have $\log_2\#\Sel(E_D)\ll\log_2\#\Sel_2(E) + O(w_D)$, where
$w_D$ is the number of prime factors of $D$. Therefore, we have the
upper bounds
\begin{equation}\label{regulatorbound}
r_{E/K}\ll_E r_D \ll_E \left(\frac{\log  |D|}{\log \log |D|}\right)
\end{equation}
on the ranks of $E_D/\Q$ and $E_K$.

Assume that we have the subconvex bound $|L(E_D/\Q,s)|\ll_{E,\epsilon}|D|^{1/2+|s|-\delta+\epsilon}$ for some $\delta>0$. This immediately implies the bound $|L(E/K,s)|\ll_{E,\epsilon}|\Disc(K)^{1/2}|^{1/2+|s|-\delta+\epsilon}$. Theorem \ref{thm: mainquadratic} now follows immediately from Lemma \ref{lem:Lsubconvex} by an argument identical to the proof of Theorem \ref{thMain5}.
$\Box$


\end{document}